\documentclass{amsart}
\usepackage{amssymb}
\usepackage{amscd}
\usepackage{graphicx}
\usepackage{amsmath}
\usepackage{setspace}
\usepackage[all]{xy}
\usepackage{color}
\usepackage{enumerate}

\newtheorem{theorem}{Theorem}[section]

\newtheorem{corollary}[theorem]{Corollary}

\newtheorem{lemma}[theorem]{Lemma}

\newtheorem{proposition}[theorem]{Proposition}

\theoremstyle{remark}
\newtheorem{definition}[theorem]{Definition}
\newtheorem{example}[theorem]{Example}
\newtheorem{remark}[theorem]{Remark}

\newtheorem{notation}[theorem]{Notation}

\newcommand{\CC}{\mathcal{C}}

\newcommand{\BBQ}{{\mathbb Q}}

\newcommand{\Z}{{\mathbb Z}}

\newcommand{\Hom}{\mathrm{Hom}}

\newcommand{\Ker}{\mathrm{Ker}}

\newcommand{\End}{\mathrm{End}}

\newcommand{\Add}{\mathrm{ Add}}
\newcommand{\add}{\mathrm{ add}}

\newcommand{\Prod}{\mathrm{ Prod}}

\newcommand{\Modr}{\mathrm{ Mod}\text{-}}

\newcommand{\Ab}{\text{\bf Ab}}

\begin{document}


\title[The Krull-Remak-Schmidt-Azumaya Theorem]{The Krull-Remak-Schmidt-Azumaya Theorem for idempotent complete categories}

\author{Simion Breaz}

\address{Simion Breaz: "Babe\c s-Bolyai" University, Faculty of Mathematics and Computer Science, Str. Mihail Kog\u alniceanu 1, 400084, Cluj-Napoca, Romania}

\email{simion.breaz@ubbcluj.ro; bodo@math.ubbcluj.ro}





\begin{abstract}
We prove a version of the Krull-Remak-Schmidt-Azumaya unique decomposition theorem in idempotent complete additive categories. In particular, this result applies to compactly generated triangulated categories. In addition, we provide an example of an object with the finite exchange property that does not have the exchange property.
\end{abstract}

\maketitle

\section{Introduction}

The Krull-Remak-Schmidt-Azumaya Theorem, proved in \cite{Azu}, states that if a module $M$ has a decomposition $M=\bigoplus_{i\in I}M_i$ as a direct sum of submodules with local endomorphism rings then \begin{enumerate}[{\rm a)}]
    \item every indecomposable direct summand of $M$ is isomorphic to a submodule $M_i$, and 
    \item the above decomposition is unique up to isomorphism and a permutation of the direct summands.\end{enumerate} For historical details on the development of this theorem and its applications in module categories, we refer to \cite{Fac}.

Similar statements have also been proved and applied to more general categories. A version of statement b) for finite decompositions in additive categories is proven in \cite[Theorem 1]{WW}. For a), it is necessary to assume that all idempotents split (a category with this property is called \textit{idempotent complete}). In fact, if every object has a decomposition as a finite direct sum of objects with local endomorphism rings, then the category is idempotent complete, as shown in \cite[Corollary 4.4]{kr-ks}. For the case of infinite sums the statements a) and b) do not hold in additive categories, even if these categories have kernels (see Example \ref{ex:idempotent complete-local-non-exchange} and Example \ref{ex:non-HSA-general-additive}). However, it is proved in \cite[Theorem 2 and Theorem 3]{WW} that the Krull-Remak-Schmidt-Azumaya Theorem is valid in additive categories with kernels that satisfy a Grothendieck condition. For other results involving lattices with a Grothendieck condition, we refer to \cite{HR}.

Recently, certain versions of a) and b) have been used, for instance, in \cite{BR} or \cite{NV}, in idempotent complete categories without kernels. It seems that establishing a general version of the Krull-Remak-Schmidt-Azumaya Theorem in the context of such categories could be useful.

This is the main aim of the present note. We begin with a brief study of objects with the (finite) exchange property in idempotent complete additive categories. In particular, we prove in Proposition \ref{prop:finite-exchange-prop} that an indecomposable object in an idempotent complete additive category has the finite exchange property if and only if its endomorphism ring is local. Proposition \ref{prop:fin-ex-global} shows that every object with local endomorphism ring which satisfies a suitable hypothesis, generalizing finitely approximable or weak Grothendieck conditions used in \cite{WW}, have the exchange property. Furthermore, in Example \ref{ex:idempotent complete-local-non-exchange}, we show that, in the dual category $\Ab^{op}$ of the category of abelian groups, the group of $p$-adic integers has the finite exchange property, but not the exchange property. This gives a negative answer to the question presented in \cite{HR} after Definition 7.3.2.  

Then we will prove some versions of statements a) and b) which are valid in idempotent complete categories, Theorem \ref{thm:ksa-infinite}. Consequently, the Krull-Remak-Schmidt-Azumaya Theorem holds for idempotent complete categories with enough compact objects and, in particular, for compactly generated triangulated categories, as shown in Corollary \ref{cor:triang}.  

\section{Preliminaries}

For the reader's convenience, in this section we recall fundamental properties and concepts related to direct sum decompositions in additive categories and establish the notation used throughout the paper.

Let $\CC$ be an additive category. Recall from \cite[Proposition I.18.1]{Mi} that an object $A$ in $\CC$ has a decomposition $A=B\oplus C$ with the canonical morphisms $u:B\to A$ and $v:C\to A$ if and only if there exist morphisms $p:A\to B$ and $q:A\to C$ such that $pu=1_B$, $qv=1_C$, $qu=0$, $pv=0$, and $up+vq=1_A$. It is easy to check that $q$ is the cokernel of $u$ and $p$ is the cokernel of $v$. Therefore, they are unique up to isomorphism of quotient objects, and we will call them \textit{the canonical projections} associated with the decomposition $A=B\oplus C$. 

\begin{definition} Let $\CC$ be an additive category and $u:B\to A$ a morphism in $\CC$.
\begin{enumerate}[a)]
\item We will say that $u:B\to A$ represents the \textit{direct summand} $B$ of $A$ if there exists a morphism $v:C\to A$ such that $u$ and $v$ are the canonical morphisms for a direct decomposition $A=B\oplus C$. In this case, $C$ is called a \textit{complement} for $B$.
\item The morphism $u$ is a \textit{split monomorphism} if there exists $p:A\to B$ such that $pu=1_B$.
\end{enumerate}
\end{definition}

From the preceding discussion, it follows that two complements for the same direct summand are isomorphic as objects (but not necessarily as subobjects of $A$). Moreover, let us remark that if $u$ represents a direct summand then it is a split monomorphism, but the converse of this implication is not true in general. From \cite[Proposition I.18.5]{Mi} it follows that a split monomorphism $u:B\to A$ represents a direct summand if and only if there exists $p:A\to B$ such that $pu=1_B$ and the idempotent endomorphism $pu$ of $B$ has a kernel. 

\begin{definition}
An additive category $\CC$ is said to be \textit{idempotent-complete} if every idempotent endomorphism in $\CC$ has a kernel.   
\end{definition}

We recall the following characterizations and properties:

\begin{lemma}\label{lem:idempotent-complete}
If $\CC$ is an additive category, the following are equivalent:
\begin{enumerate}[{\rm a)}]
    \item $\CC$ is idempotent-complete;
    \item every idempotent endomorphism in $\CC$ has a cokernel;
    \item for every $A\in \CC$ and every idempotent endomorphism $e$ of $A$ there exist morphisms $q:B\to A$ and $p:A\to B$ such that $pq=1_B$ and $qp=e$.
\end{enumerate}
Moreover, under these conditions, the following are true:
\begin{enumerate}
    \item[{\rm d)}] \cite[Proposition I.18.5]{Mi} If $e$ is an idempotent endomorphism of $A$ then $A=\Ker(e)\oplus \Ker(1-e)$, where the canonical morphisms are those associated to these kernels.
    \item[{\rm e)}] \cite[Lemma I.3.4]{Bass} If $k:A\to B$ and $u:B\to A$ are morphisms such that $ku$ is an automorphism of $B$ then $u$ represents a direct summand of $A$. 
\end{enumerate}
\end{lemma}

\begin{remark}
In an idempotent-complete additive category, a monomorphism represents a direct summand of an object if and only if it is a split monomorphism. When we refer to a direct summand we often understand that it comes together with a split monomorphism, and that it is the domain of this monomorphism.      
\end{remark}

In the following, we will use the following notations:

\begin{notation}\label{notation-direct-sums-summands}
Let $X=\bigoplus_{i\in I}A_i$ be a direct sum in $\CC$ with the canonical morphisms $u_i:A_i\to X$. In this case, we will say that \textit{the family $(u_i)_{i\in I}$ represents a direct decomposition of $X$}. If $J\subseteq I$, we denote $X(J)=\bigoplus_{j\in J}A_j$, and $u_J:X(J)\to X$ will be the morphism induced by the family $u_j, j\in J$. Moreover, we will denote by $p_J:X\to X(J)$ the canonical projection associated to the direct decomposition $X=X(J)\oplus X(I\setminus J)$. Hence, $p_J$ is the cokernel of $u_{I\setminus J}$. Moreover, let us observe that $p_Ju_j$, $j\in J$, are the canonical morphisms of the direct decomposition $X(J)=\bigoplus_{j\in J}A_j$.    
\end{notation}

\begin{notation}
For the morphisms $f:A_1\oplus A_2\to B_1\oplus B_2$, we will use the (standard) matrix notations $f=\left[\begin{array}{cc}
   \alpha_{11}  & \alpha_{12} \\
    \alpha_{21} & \alpha_{22}
\end{array}\right],$ where $\alpha_{ij}=p_i f u_j:A_j\to B_i$ are the morphism induced by the canonical morphisms $u_i:A_i\to A_1\oplus A_2$ and $p_j:B_1\oplus B_2\to B_j$. For more details, we refer to \cite[Section I.3]{Bass} and \cite[p.26]{Mi}.   
\end{notation}


\section{The exchange property in idempotent complete additive categories}

The exchange property was introduced in \cite{CJ} for general algebraic systems and in \cite{WW} for additive categories. 

\begin{definition}
   Let $\CC$ be an additive category. An object $A\in \CC$ has the \textit{exchange property with respect to a direct decomposition $B=\bigoplus_{i\in I}B_i$} (with the canonical monomorphisms $u_i:B_i\to B$) if for any split monomorphism $u:A\to B$, there exist direct decompositions $B_i=C_i\oplus D_i$ with the canonical monomorphisms $v_i$ and $w_i$ such that $B=A\oplus \left(\bigoplus_{i\in I}C_i\right)$ with the canonical monomorphisms $u$ and $u_iv_i$. If we know that this property is valid for all (finite) direct decompositions of the objects of $\CC$ then we will say that $A$ has the (\textit{finite}) \textit{exchange property}.  
\end{definition}

\begin{remark}\label{rem:KS-first} 
Suppose that the object $A\in \CC$ has the exchange property with respect to a direct decomposition $B=\bigoplus_{i\in I}B_i$ and that $D_i$ are the direct summands used in the above definition. Then $A$ and $\bigoplus_{i\in I}D_i$ are complements of the same direct summand $\bigoplus_{i\in I}C_i$ of $B$, hence $A\cong \bigoplus_{i\in I}D_i$. If $A$ is indecomposable, then there exists $i_0\in I$ such that $D_i=0$ for all $i\neq i_0$.   
\end{remark}

The following lemma is proved in \cite{war} (see also \cite{kr-ks}). 

\begin{lemma}\label{lem:warfield-equiv}
Let $\CC$ be an idempotent complete additive category. If $A$ is an object of $\CC$ and $E$ is the endomorphism ring of $A$ then the functor $\Hom(A,-):\add(A)\to \mathrm{proj}(E)$ is an equivalence, where $\add(A)$ is the full subcategory of $\CC$ of all objects which are direct summands of finite direct sums of copies of $A$, and $\mathrm{proj}(E)$ represents the category of all finitely generated projective right $E$-modules.     
\end{lemma}

\begin{definition}
 A ring $R$ is an \textit{exchange ring} if the regular left $R$-module $R$ has the finite exchange property.   
\end{definition}

Note that this notion is a left-right symmetric property. An $R$-module $M$ has the finite exchange property if and only if its endomorphism ring is exchange, \cite{W72}. Exchange rings can be characterized by using ring theoretic properties, see \cite[Theorem 1]{Monk} and \cite[Theorem 2.1]{Ni}. From the proofs of these two theorems we deduce the following characterization: 

\begin{lemma}\label{lem:exchange-M+M}
An $R$-module $M$ has the finite exchange property if and only if it has the exchange property with respect to all direct sums $M_1\oplus M_2$ with $M_1=M=M_2$. 
\end{lemma}


In the next result, we use Lemma \ref{lem:warfield-equiv} to characterize the objects that satisfy the finite exchange property in idempotent-complete additive categories.

\begin{proposition}\label{prop:finite-exchange}
Let $\CC$ be an idempotent complete additive category. If $A$ is an object from $\CC$ and $E$ is its endomorphism ring, the following are equivalent:
\begin{enumerate}[{\rm a)}]
    \item $A$ has the finite exchange property;
    \item $E$ is an exchange ring. 
\end{enumerate}
\end{proposition}

\begin{proof}
a)$\Rightarrow$b) Let $H^A=\Hom(A,-):\CC\to \Modr E$ be the covariant $\Hom$ functor induced by $A$. From Lemma \ref{lem:warfield-equiv} the restriction $H^A:\add(A)\to \mathrm{proj}(E)$ is an equivalence. To prove that the regular module $E$ has the finite exchange property, we will use Lemma \ref{lem:exchange-M+M} to observe that it is enough to prove that if $G=E\oplus E$ with the canonical monomorphisms $u_1,u_2:E\to G$ and we have a direct decomposition $G=M\oplus N$ (of $E$-modules) with canonical monomorphisms $u_M$ and $u_N$ and $M\cong E$ then there exist split monomorphisms $v_{1}:E_1\to E$ and $v_{2}:E_2\to E$ such that $G=M\oplus E_1\oplus E_2$ with canonical monomorphisms $u_M$, $u_1v_1$, and $u_2v_2$.

Since $G$, $M$ and $N$ are in $\mathrm{proj}(E)$, we can assume that there exists in $\CC$ a direct decomposition $B\oplus C=A\oplus A$ such that $B\cong A$, $\Hom(A,C)=N$, and for the canonical morphisms $\overline{u}_B$ and $\overline{u}_C$, associated to the decomposition $B\oplus C=A\oplus A$, we have  $\Hom(A,\overline{u}_B)=u_M$ and $\Hom(A,\overline{u}_C)=u_N$. 

Therefore, we can apply the finite exchange property for $A$ to find two split monomorphisms $\overline{v}_i:D_i\to A$, $i=1,2$, such that $\overline{u}_B$, $\overline{u}_1\overline{v}_1$ and $\overline{u}_2\overline{v}_2$ give a direct decomposition for $A\oplus A$ ($\overline{u}_1$ and $\overline{u}_2$ are the canonical morphisms $A\to A\oplus A$). Applying $\Hom(A,-)$ we obtain the desired decomposition for $E\oplus E$, hence $E$ has the finite exchange property. 

b)$\Rightarrow$a) Assume that $I$ is a finite set, $B=\bigoplus_{i\in I} B_i$ is an object from $\CC$, and that $u:A\to B$ is a split monomorphism. We consider the equivalence $\Hom(B,-):\add(B)\to \mathrm{proj}(\End(B))$. It follows that the endomorphism ring of $\Hom(B,A)$ is isomorphic to $E$, hence the $\End(B)$-module $\Hom(B,A)$ has the finite exchange property. Then we can apply the same strategy as in the proof of a)$\Rightarrow$b) to obtain the desired decomposition for $B$. 
%
\end{proof}

From \cite[Proposition 1]{W69} we obtain the following characterization.

\begin{proposition}\label{prop:finite-exchange-prop}
Let $\CC$ be an idempotent complete additive category. An indecomposable object $A$ from $\CC$ has the finite exchange property if and only if $\End(A)$ is a local ring.
\end{proposition}

\begin{corollary}\cite[Lemma 3.5]{Bass}\label{cor:direct-summand-local}
Suppose that $\CC$ is an idempotent complete additive category. Let $X=\bigoplus_{i=1}^n A_i$ be an object in $\CC$ and $B$ a direct summand in $X$. 
 
If $\End(B)$ is a local ring then there exists $j\in\{1,\dots,n\}$ and a direct summand $A_j'$ of $A_j$ such that $X=B\oplus A'_j\oplus\left(\bigoplus_{i\neq j}A_i\right)$.   \end{corollary}

\begin{proposition}\label{prop:finite-exchange-prop-ds}
Let $\CC$ be an additive category. 
\begin{enumerate}[{\rm a)}]
    \item A finite direct sum of objects with the (finite) exchange property has the (finite) exchange property.
\item If $\CC$ is idempotent complete, a direct summand of an object with the finite exchange property has the finite exchange property.
\end{enumerate}
\end{proposition}

\begin{proof}
a) This is proved in \cite[Lemma 6]{WW}.

b) From \cite[Lemma 2.4]{Fac}, in every module category, the class of all modules with the finite exchange property is closed with respect to direct summands. Therefore, if $e$ is an idempotent in an exchange ring $E$ then $eEe$ is also an exchange ring. Applying Proposition \ref{prop:finite-exchange}, the conclusion is now obvious.
\end{proof}

\begin{corollary}\label{cor:ds-sum-locale}\label{thm:ks-sd-locale}
Suppose that $\CC$ is an idempotent complete additive category. Let $X=\bigoplus_{i=1}^n A_i=B\oplus C$ be an object in $\CC$ such that $B$ is indecomposable. 
If the rings $\End(A_i)$ are local then there exists $j\in\{1,\dots,n\}$  such that $X=B\oplus\left(\bigoplus_{i\neq j}A_i\right)$ (hence $B\cong A_j$).    \end{corollary}

\begin{proof}
We use Proposition \ref{prop:finite-exchange-prop-ds} to observe that $B$ has the finite exchange property. Since it is indecomposable. From Proposition \ref{prop:finite-exchange-prop}, it follows that $\End(B)$ is local. Now we apply Corollary \ref{cor:direct-summand-local}.    
\end{proof}





{

In order to provide a sufficient condition for objects with local endomorphism rings to have the exchange property, we consider a generalized version of the weak Grothendieck condition introduced in \cite{WW}.

\begin{definition}
Let $\CC$ be an additive category. We will say that an object $B$ from $\CC$ satisfies a \textit{non-split Grothendieck condition} if for every direct decomposition $X=\bigoplus_{i\in I}A_i$ and for every split monomorphism $u:B\to X$ there exists a finite subset $J\subseteq I$ such that the morphism $[u,u_J]:B\oplus X(J)\to X$ is not a split monomorphism.   
\end{definition}

In order to give some examples, we need the following lemma.

\begin{lemma}\label{lem:non-split}
Let $\CC$ be an additive category. If $u:B\to X$ and $v:D\to X$ are morphisms in $\CC$ such that there exist $f:C\to B$ and $g:C\to D$ with $uf=vg\neq 0$ then $[u,v]:B\oplus D\to X$ is not a split monomorphism.    
\end{lemma}

\begin{proof}
Suppose that $[u,v]$ is a split monomorphism, it follows that there exist $p:X\to B$ and $q:X\to D$ such that for the induced morphism $\left[\begin{array}{c}
    p \\ q
\end{array}\right]:X\to B\oplus D$ we have $\left[\begin{array}{c}
    p \\ q
\end{array}\right]\left[u,v\right]=1_{B\oplus D}$. Then  $qu=0$ and $qv=1_{D}$. It follows that $g=qvg=quf=0$, a contradiction.   
\end{proof}

\begin{example}\label{rem:weakG-is-nonsplitG}
Every object that satisfies a weak Grothendieck condition used in \cite{WW} also satisfies a non-split Grothendieck condition.
%
\end{example}

\begin{proof}
Recall from \cite[p. 353]{WW} that an object $B\neq 0$ satisfies a weak Grothendieck condition if for every monomorphism $u:B\to X$, with $X=\bigoplus_{i\in I}A_i$, there exist a non-zero morphism $f:C\to B$, a finite subset $J\subseteq I$ and a (non-zero) morphism $g:C\to X(J)$ such that $uf=u_Jg$. In these conditions, it follows, from Lemma \ref{lem:non-split}, that the morphism $[u,u_J]:B\oplus X(J)\to X$ is not a split monomorphism. 
\end{proof}

\begin{definition}
An object $C$ in an additive category $\CC$ with direct sums is called \textit{compact} if for every morphism $f:C\to \bigoplus_{i\in I}A_i$ there exists a finite subset $J\subseteq I$ such that $f$ factorizes through $u_J$. 
\end{definition}

\begin{example}
Suppose that $\CC$ is an additive category and that $B$ is an object such that there exists a non-zero morphism $C\to B$ with $C$ a compact object. Then $B$ satisfies a non-split Grothendieck condition. In particular, every finitely approximable non-zero object satisfies a non-split Grothendieck condition. \end{example}

\begin{proof}
For the first statement, consider a split monomorphism $u:B\to X$, with $X=\bigoplus_{i\in I}A_i$. If $f:C\to B$ is a non-zero morphism and $C$ is compact, it follows that $uf$ factorizes through a canonical morphism $u_J:X(J)\to X$ with $X$ a finite set. The conclusion follows from Lemma \ref{lem:non-split}. For the second one, it is enough to observe that in the proof of \cite[Lemma 4]{WW} it is proved that if $B\neq 0$ is a finitely approximable object, there exists a compact object $C$ and a non-zero morphism $C\to B$.    
\end{proof}



We need the following lemma.

\begin{lemma}\label{lem:M+N=K+L}
Let $\CC$ be an idempotent complete category. Suppose that $X$ is an object of $\CC$, and that $X=M\oplus N=K\oplus L$ are direct decompositions for $X$, with the canonical monomorphisms $u_M:M\to X$ and $u_K:K\to X$, such that $M$ is indecomposable and $[u_M,u_K]:M\oplus K\to X$ is not a split monomorphism.
\begin{enumerate}[{\rm a)}]
    \item If $\End(M)$ is a local ring and $K=\bigoplus_{j\in J}K_j$, with $J$ a finite set, then there exist split monomorphisms $v_j:K'_j\to K_j $, $j\in J$, such that $(u_M,(u_Kv_j)_{j\in J},u_L)$ represent a direct decomposition for $X$.
    \item If $K$ has the finite exchange property then $M$ is isomorphic to a direct summand of $K$. 
\end{enumerate}
\end{lemma}

\begin{proof}
 a) Using Proposition \ref{prop:finite-exchange-prop}, we observe that $M$ has the finite exchange property. It follows that there are direct decompositions $K_j=K'_j\oplus K''_j$ and $L=L_1\oplus L_2$ (with canonical monomorphisms $v_j:K'_j\to K$ and $w_i:L_i\to L$, $i=1,2$) such that we have a direct decomposition $X=M\oplus (\bigoplus_{j\in J} K'_j)\oplus L_1$ given by the canonical morphisms $u_M$, $u_Kv_j$, and $u_Lw_1$ (here $u_L:L\to X$ is the canonical monomorphism). Since the morphisms $u_Kv_j$ and $u_Lw_1$ are also the canonical split monomorphism associated to the direct decomposition $X=(\bigoplus_{j\in J}K'_j)\oplus (\bigoplus_{j\in J}K''_j)\oplus L_1\oplus L_2$, it follows that there exists an isomorphism $M\cong (\bigoplus_{j\in J}K''_j)\oplus L_2$. Since $M$ is indecomposable, we obtain $\bigoplus_{j\in J}K''_j=0$ or $L_2=0$. If $(\bigoplus_{j\in J}K''_j)=0$,  it follows that $[u_M, u_K]$ is a split monomorphism, a contradiction. Hence, $L_2=0$ and the conclusion is obvious.

 b) Since $K$ has the finite exchange property, there exist direct decompositions $M=M_1\oplus M_2$ and $N=N_1\oplus N_2$ such that $X=K\oplus M_1\oplus N_1$. As in the proof of a), if $M_2=0$, it follows that $[u_M,u_K]$ is a split monomorphism, a contradiction. Therefore, $M_1=0$, hence $X= K\oplus N_1=M\oplus N_1\oplus N_2$. In these decompositions the canonical monomorphisms associated to $N_1$ are the same, hence $K\cong M\oplus N_2$.  
\end{proof}

The following result generalizes \cite[Proposition 2]{WW}.

\begin{proposition}\label{prop:fin-ex-global}
Let $\CC$ be an idempotent complete additive category. If $A\in \CC$  satisfies a non-split Grothendieck condition and the endomorphism ring of $A$ is local then $A$ has the exchange property.    
\end{proposition}

\begin{proof}
Consider a split monomorphism $u:A\to \bigoplus_{i\in I}B_i$. Using the hypothesis, we observe that there exists a finite subset $J\subseteq I$ such that $[u,u_J]$ does not split. If we apply Lemma \ref{lem:M+N=K+L} a) for $K=X(J)$ and $L=X(I\setminus J)$, we obtain the conclusion.
\end{proof}

In the following example, we will use an idea from \cite[Example 2.1]{Fr} to prove that Proposition \ref{prop:fin-ex-global} does not hold for arbitrary objects in idempotent complete categories.

\begin{example}\label{ex:idempotent complete-local-non-exchange}
Let $\Ab$ be the category of all abelian groups. If $p$ is a prime, we denote by 
$$\widehat{\Z}_p=\{(a_n+\Z/p^n\Z)_{n>0}\mid a_{n+1}-a_n\in p^n\Z\}$$ the group of $p$-adic integers. Following the same strategy as in the proof of \cite[Lemma 6.4.6]{fuchs-15}, we will prove that $\widehat{Z}_p$ is a pure subgroup in $\prod_{n>0}\Z/p^n\Z$. Recall from \cite[Section 5.1]{fuchs-15} that a subgroup $H$ of $G\in \Ab$ is pure if every positive integer $m$ we have $mH=H\cap(mG)$. Since our groups are $q$-divisible for all $q$ with $\gcd(q,p)=1$, we have to verify the above equality only for the case $m=p^k$ with $k$ an arbitrary positive integer. In this case, it is easy to see that a $p$-adic integer $a=(a_n+\Z/p^n\Z)_{n>0}$ does not belong to $p^k\widehat{\Z}_p$ if and only if there exists an integer $n>0$ such that $a_n+\Z/p^n\Z\notin p^k(\Z/p^n\Z)$. Therefore, if $a\notin p^k\widehat{\Z}_p$ then $a\notin p^k\prod_{n>0}\Z/p^n\Z$. It follows that $\widehat{\Z}_p$ is a pure subgroup of $\prod_{n>0}\Z/p^n\Z$. But $\widehat{\Z}_p$ is pure-injective (see \cite[Example 6.4.3]{fuchs-15}), hence there exists $u:\prod_{n>0}\Z/p^n\Z\to \widehat{\Z}_p$ such that $u_{|\widehat{\Z}_p}=1_{\widehat{\Z}_p}$. We note, for the reader's convenience, that a similar conclusion can be obtained by using \cite[Theorem 6.4.7]{fuchs-15}, but in this case the direct product also contains copies of the Pr\"ufer group $\Z(p^\infty)$.   

If we denote by $u^{op}$ the morphism induced by $u$ in the opposite category $\Ab^{op}$, it follows that $u^{op}$ is a split monomorphism in $\Ab^{op}$ from $\widehat{\Z}_p$ to the direct sum (in $\Ab^{op}$) of the family of indecomposable objects $(\Z/p^n\Z)_{n>0}$. Since the group of $p$-adic integers is not isomorphic (in $\Ab$) to a direct product of cyclic $p$-groups, it follows that $\widehat{\Z}_p$ does not have the exchange property in $\Ab^{op}$. However, the endomorphism ring of $\widehat{\Z}_p$ is the (commutative) ring of all $p$-adic integers, hence it is local.  
\end{example}

\begin{remark}\label{rem:non-exchange-local}
This example provides a negative answer to the open question of whether, in bounded lattices (in particular, in the subobject lattice of an object in abelian categories), every element with the finite exchange property also has the full exchange property, as discussed in \cite[p. 61]{HR}. The category $\Ab^{op}$ contains a generator, the injective cogenerator $\BBQ/\Z$ in $\Ab$, and satisfies the Grothendieck conditions AB4 and AB4$^\star$, but it is not a Grothendieck category. The question remains unresolved in Grothendieck categories, particularly in module categories, and we refer to \cite{IGS} for recent progress in the study of the exchange property in module categories. 
\end{remark}
}

\section{The Krull-Remak-Schmidt-Azumaya Theorem in Additive Categories}





We are ready to prove the main result of this note, a (uniqueness) Krull-Remak-Schmidt-Azumaya Theorem for idempotent complete categories.


\begin{theorem}\label{thm:ksa-infinite}
Suppose that $\CC$ is an idempotent complete additive category. 
Let $(A_i)_{i\in I}$ be a family of objects from $\CC$ with local endomorphism rings.
\begin{enumerate}[{\rm a)}]
    \item  If $B$ is an indecomposable direct summand of $\bigoplus_{i\in I}A_i$ {which satisfies a non-split Grothendieck condition} then there exists $i\in I$ such that $B\cong A_i$.
    \item If $\bigoplus_{i\in I}A_i\cong \bigoplus_{j\in J}B_j$, all objects $B_j$ are indecomposable, and all objects $A_i$ $(i\in I)$ and $B_j$ $(j\in J)$ satisfy non-split Grothendieck conditions, then there exists a bijection $\sigma:I\to J$ such that for every $i\in I$ we have $A_i\cong B_{\sigma(i)}$.
\end{enumerate}
\end{theorem}

\begin{proof}
The proof follows the classical proof presented in \cite{WW}. We give some details for the reader's convenience. We will use the notations described in Notation \ref{notation-direct-sums-summands} for $X=\bigoplus_{i\in I}A_i$.

a) Suppose that $u_B:B\to X$ is a split monomorphism, and fix a finite subset $J\subseteq I$ such that $[u_B,u_J]:B\oplus X(J)\to X$ is not a split monomorphism. From Proposition \ref{prop:finite-exchange-prop} it follows that all objects $A_i$ have the finite exchange property.Therefore, using Proposition \ref{prop:finite-exchange-prop-ds}, we conclude that $X(J)$ has also the finite exchange property. From Lemma \ref{lem:M+N=K+L} b), it follows that $B$ is isomorphic to a direct summand of $X(J)$. 
The conclusion follows from Corollary \ref{thm:ks-sd-locale}.

b) 
We can assume that we have two direct decompositions $X=\bigoplus_{i\in I}A_i= \bigoplus_{j\in J}B_j$, and we denote by $u_i:A_i\to X$, $i\in I$, and $v_j:B_j\to X$, $j\in J$, the canonical morphisms which induce these decompositions. 

From a) it follows that for every $j\in J$ there exists $i\in I$ such that $A_i\cong B_j$. 
It follows that all objects $B_j$ have local endomorphism rings. Consequently, we can apply Proposition \ref{prop:finite-exchange-prop} to conclude that all objects $A_i$ and  $B_j$ have the exchange property. 

Using again a) we conclude that for every $i\in I$ there exists $j\in J$ such that $B_j\cong A_i$.  
It follows that we have a bijection between the sets $\{[A_i]\mid i\in I\}$ and $\{[B_j]\mid j\in J\}$, where $[Y]$ denotes the isomorphism class of $Y$. 

We obtain two partitions $I=\bigcup_{k\in K}I_k$ and $J=\bigcup_{k\in K}J_k$ with the following properties: 
\begin{itemize}
\item two objects $A_i$ and $A_{i'}$ are isomorphic if and only if there exists $k\in K$ such that $i,i'\in I_k$;
\item two objects $B_j$ and $B_{j'}$ are isomorphic if and only if there exists $k\in K$ such that $j,j'\in J_k$;
\item if $k\in K$, $i\in I_k$, $j\in J_k$ then $A_i\cong B_j$.
\end{itemize}

We want to prove that for all $k\in K$ we have $|I_k|=|J_k|$. It is enough to prove that $|J_k|\leq |I_k|$.

\textsl{Case 1. $J_k$ is finite.} For this case we use an induction on $|J_k|$ to prove directly that $|J_k|\leq |I_k|$. For $|J_k|=0$ the inequality is obvious. If $j_0\in J_k$, since $B_{j_0}$ has the exchange property, we can apply Remark \ref{rem:KS-first} to observe that there exists $i_0\in I_k$ such that $(v_{j_0},(u_i)_{i\neq i_0})$ represents a direct decomposition of $X$. Then $\bigoplus_{i\in I\setminus\{i_0\}}A_i\cong \bigoplus_{j\in J\setminus\{j_0\}}B_j$, and we can apply the induction hypothesis.

\textsl{Case 2. $J_k$ is infinite.} Using a strategy similar to that used in Case 1, we observe that $I_k$ is infinite. We leave the remaining proof to the reader, since it is identical with that presented in \cite[Theorem 3]{WW}.
\end{proof}

\begin{definition}
The category $\CC$ is said to have \textit{enough compact objects} if for every non-zero object $B\in \CC$ there exist a compact object $C\in \CC$ and a non-zero morphism $C\to B$.
\end{definition}

\begin{definition} Recall that a triangulated category $\CC$ is \textit{compactly generated} if it has direct sums, has enough compact objects, and the isomorphism classes of the compact objects form a set. 
\end{definition}

\begin{corollary}\label{cor:ksa-enough-compact}\label{cor:triang}
Suppose that $\CC$ is an idempotent complete category with enough compact objects and that $(A_i)_{i\in I}$ is a family of objects from $\CC$ with local endomorphism rings.    
\begin{enumerate}[{\rm a)}]
    \item  For every indecomposable direct summand $B$ of $\bigoplus_{i\in I}A_i$ there exists $i\in I$ such that $B\cong A_i$.
    \item If $\bigoplus_{i\in I}A_i\cong \bigoplus_{j\in J}B_j$ and all objects $B_j$ are indecomposable then there exists a bijection $\sigma:I\to J$ such that for every $i\in I$ we have $A_i\cong B_{\sigma(i)}$.
\end{enumerate}
\end{corollary}

\begin{remark}
It follows from \cite[Proposition 1.6.8]{Ne} that every compactly generated triangulated category is idempotent complete. Hence, the hypothesis of Corollary \ref{cor:ksa-enough-compact} holds whenever $\CC$ is a compactly generated triangulated category.     
\end{remark}


We close the paper with an example that shows that Theorem \ref{thm:ksa-infinite} does not hold in general for categories without enough compact objects. This is based on \cite[Example 2.1]{Fr}. 

\begin{example}\label{ex:non-HSA-general-additive} If $p$ is a prime then the group of $p$-adic integers $\widehat{\Z}_p$ is the inverse limit of the system of the canonical epimorphisms $p_n:\Z/p^{n+1}\Z\to \Z/p^{n}\Z$, $n>0$. In fact, we have a short exact sequence 
$$0\to \widehat{\Z}_p\to \prod_{n>0 }\Z/p^{n}\Z\overset{\pi}\to \prod_{n> 0}\Z/p^{n}\Z\to 0,$$
where $(a_n+p^n\Z)_n\overset{\pi}\mapsto (a_{n+1}-a_n+p^n\Z)_n$. From Example \ref{ex:idempotent complete-local-non-exchange} it follows that this sequence splits, hence we have in $\Ab$ an isomorphism $\widehat{\Z}_p\times \prod_{n> 0}\Z/p^{n}\Z\cong \prod_{n> 0}\Z/p^{n}\Z$. Looking at this isomorphism in $\Ab^{op}$, we conclude that the statements a) and b) from Theorem \ref{thm:ksa-infinite} are not true in general idempotent complete categories.    
\end{example}

\subsection*{Funding} The revised version of this paper is supported by a grant of the Ministry of Research, Innovation
and Digitization, CNCS - UEFISCDI, project number PN-IV-P1-PCE-2023-0060, within PNCDI IV.

\end{document}